\documentclass[11pt]{amsart}
\usepackage{amssymb}
\newtheorem{lemme}{Lemma}

\newtheorem{theorem}{Theorem}
\newtheorem{remark}{Remark}

\begin{document}

\title [ Some generalizations and  refined Hardy type integral inequalities]{Some new  refined  Hardy type integral inequalities\\}%

\author[ Khaled Mehrez]{  Khaled Mehrez }
\address{Khaled Mehrez. D\'epartement de Math\'ematiques ISSATK, Kairouan 3000, Tunisia.}
 \email{k.mehrez@yahoo.fr}
 
\begin{abstract}
In this paper, by using Jensen's inequality and Chebyshev integral inequality,  some generalizations and  new refined  Hardy type integral inequalities are obtained.  In addition, the corresponding reverse relation are also proved.
\end{abstract}
\maketitle
{\it keywords:} Hardy integral inequality,  H\"older inequality, Chebyshev integral inequality, Jensen inequality.
\section{Introduction} 
In 1920, Hardy \cite{ha} proved the following inequality. If $p>1, f\geq 0,\;p-$integrable on $(0,\infty)$ and 
\begin{equation}\label{1}
F(x)=\int_{0}^{x}f(t)dt,
\end{equation}
then 
\begin{equation}\label{rrr}
\int_0^{\infty}\left(\frac{F(x)}{x}\right)^{p}<\left(\frac{p}{p-1}\right)^{p}\int_{0}^{\infty}f^{p}(x)dx,
\end{equation}
unless $f\equiv0$. The constant $\left(\frac{p}{p-1}\right)^{p}$ is the best possible. Hardy's inequality plays an important role in analysis and applications.\\
The previous inequality still holds for parameters a and b. That is, the inequality
$$\int_a^{b}\left(\frac{F(x)}{x}\right)^{p}<\left(\frac{p}{p-1}\right)^{p}\int_{a}^{b}f^{p}(x)dx,$$
is valid for $0<a<b<\infty,$ see \cite{lg}.\\

In \cite{ws}, the author gives a generalization and improvement for Hardy's inequality in the sens when $f$ is non-decreasing. If $f\geq0,$ and non-decreasing, $F$ is as defined by (\ref{1}). $f\geq0,\, g>0,\,\frac{x}{g(x)}$ is non-increasing,~$p>1,\,0<a<1$  then 
\begin{equation}
\int_{0}^{\infty}\left(\frac{F(x)}{g(x)}\right)^{p}dx\leq\frac{1}{a(1-p)(1-a)^{p-1}}\int_{0}^{\infty}\left(\frac{xf(x)}{g(x)}\right)^{p}dx.
\end{equation}

 

In this paper,  by using Jensen's inequality and Chebyshev integral inequality, we first give a generalization of Hardy's inequality (\ref{rrr}). Second, we prove some new refined Hardy type integral inequalities.

\section{Some preliminary lemmas}
In this section, we state the following lemmas, which are useful in the proofs of our results.
\begin{lemme}\label{009}\cite{jj}(Jensen inequality)
Let $\mu$ be a probability measure and let $\varphi\geq0$ be a convex function. Then, for all $f$ be a integrable function we have 
$$\int \varphi\circ f d\mu\geq \varphi\left(\int f d\mu\right).$$
\end{lemme}
\begin{lemme}\cite{mi}\label{mi}(Chebyshev integral inequality)
If $f,g:[a,b]\longrightarrow\mathbb{R}$ are integrable functions, both increasing or both decreasing, and $p:[a,b]\longrightarrow\mathbb{R}$ is a positive integrable function, then
\begin{equation}\label{w1}
\int_{a}^{b}p(t)f(t)dt\int_{a}^{b}p(t)g(t)dt\leq\int_{a}^{b}p(t)dt\int_{a}^{b}p(t)f(t)g(t)dt
\end{equation}
Note that if one of the functions $f$ or $g$ is decreasing and the other is increasing, then (\ref{w1}) is reversed.
\end{lemme}

A function $\varphi$ is called submultiplicative, if $\varphi(xy)\leq\varphi(x)\varphi(y),$ for all $x,y>0$ In particular, for all $n\geq 1$, we have $\varphi(x^n)\leq\varphi^{n}(x)\;x>0.$

The next lemma exist in \cite{ws}.
\begin{lemme}\label{l3} 
Let $\varphi\geq0$ is submultiplicative, and $\varphi(0)=0.$ If $\varphi^{'}(x)$ is non-decreasing (non-increasing), then $\frac{\varphi(x)}{x}$ is non-decreasing (non-increasing).
\end{lemme}
\begin{lemme}\label{lll}
Let $g>0,$ and $G$ is as defined by (\ref{001}). If the function $\frac{g(x)}{x}$
is non-decreasing (non-increasing) on $(0,\infty)$, then the function $\frac{G(x)}{x^{2}}$
is also non-decreasing (non-increasing) on $(0,\infty)$ .
\end{lemme}
\begin{proof} For all $x>0,$ and suppose that the function $\frac{g(x)}{x}$ is non-decreasing 
\begin{equation*}
\left(\frac{G(x)}{x^{2}}\right)^{'}=\frac{xg(x)-2G(x)}{x^{3}}=\frac{K(x)}{x^3},
\end{equation*}
and
$$K^{'}(x)=xg^{'}(x)-g(x)=x^{2}\left(\frac{g(x)}{x}\right)^{'}\geq0,$$
and consequently, the function $K$ is non-decreasing. Since $K(0)=0,$ then $K(x)\geq0,$ which implies $\left(\frac{G(x)}{x^{2}}\right)^{'}\geq0$. Therefore $\frac{G(x)}{x^{2}}$ is non-decreasing. 
\end{proof}
\section{Main results}
The following results gives a  generalization of Hardy's inequality.
\begin{theorem}\label{t1}
Let $f\geq0,\, g>0,\,\,0<a<1,\, p>1,\:q>\frac{p-a(p-1)}{2}$ and 
\begin{equation}\label{001}
G(x)=\int_{0}^{x}g(t)dt.
\end{equation}
If the function $\frac{x}{g(x)}$ is non-increasing function. Then
the following inequality 
\begin{equation}
\int_{0}^{\infty}\frac{F^{p}(x)}{G^{q}(x)}dx\leq\frac{1}{((a-1)(p-1)+2q-1)(1-a)^{p-1}}\int_{0}^{\infty}\frac{(tf(t))^{p}}{G^{q}(t)}dt,
\end{equation}
is valid. In particular, if we put $a=\frac{1}{p},q=\frac{p}{2}$ and $G(x)=x^2$ we obtain (\ref{rrr}).
\end{theorem}
\begin{proof} From Lemma \ref{lll}, we obtain that the function $\frac{x^{2}}{G(x)}$ is non-increasing. By using the H\"older inequality we get  
\begin{equation}
\begin{split}
\int_{0}^{\infty}\frac{F^{p}(x)}{G^{q}(x)}dx&=\int_{0}^{\infty}G^{-q}(x)\left(\int_{0}^{x}t^{\frac{-a(p-1)}{p}}t^{\frac{a(p-1)}{p}}f(t)dt\right)^{p}dx\\
&\leq\int_{0}^{\infty}G^{-q}(x)\left[\left(\int_{0}^{x}t^{-a}dt\right)^{\frac{p-1}{p}}\left(\int_{0}^{x}t^{a(p-1)}f^{p}(t)dt\right)^{\frac{1}{p}}\right]^{p}dx\\
&=\int_{0}^{\infty}G^{-q}(x)\left(\int_{0}^{x}t^{-a}dt\right)^{p-1}\int_{0}^{x}t^{a(p-1)}f^{p}(t)dtdx\\
&=\frac{1}{(1-a)^{p-1}}\int_{0}^{\infty}t^{a(p-1)}f^{p}(t)\int_{t}^{\infty}x^{(1-a)(p-1)}G^{-q}(x)dxdt\\
&\leq\frac{1}{(1-a)^{p-1}}\int_{0}^{\infty}t^{a(p-1)}f^{p}(t)\left(\frac{t^{2}}{G(t)}\right)^{q}\int_{t}^{\infty}x^{(1-a)(p-1)-2q}dxdt\\
&=\frac{1}{((a-1)(p-1)+2q-1)(1-a)^{p-1}}\int_{0}^{\infty}\frac{(tf(t))^{p}}{G^{q}(t)}dt.
\end{split}
\end{equation}
\end{proof}
The following result concerns the converse inequality.
\begin{theorem}\label{t2} Let $f\geq0,\, g>0,\,\,0<p<1,\, a>0,\:$ and $q>\frac{p+a(p-1)}{2}$.
If the function $\frac{x}{g(x)}$ is non-increasing function. Then
the following inequality 
\begin{equation}
\int_{0}^{\infty}\frac{F^{p}(x)}{G^{q}(x)}dx\geq \frac{1}{((a+1)(p-1)+2q-1)(1+a)^{p-1}}\int_{0}^{\infty}\frac{(tf(t))^{p}}{G^{q}(t)}dt.
\end{equation}
is valid. In particular, if we put $a=\frac{1}{p},\;0<p<1,\;q=\frac{p}{2}\;\textrm{and}\; G(x)=x^2$, we obtain \cite{ws}
\begin{equation}
\int_{0}^{\infty}\left(\frac{F(x)}{x}\right)^{p}dx\geq\frac{1+p}{1-p}\left(\frac{p}{1+p}\right)^{p}\int_{0}^{\infty}\left(\frac{f(x)}{x}\right)^{p}dx.
\end{equation}
\end{theorem}
\begin{proof}
By using Lemma \ref{lll}, we obtain that the function $\frac{x^{2}}{G(x)}$ is non-decreasing. From the reverse H\"older inequality we get  
\begin{equation}
\begin{split}
\int_{0}^{\infty}\frac{F^{p}(x)}{G^{q}(x)}dx&=\int_{0}^{\infty}G^{-q}(x)\left(\int_{0}^{x}t^{\frac{a(p-1)}{p}}t^{\frac{-a(p-1)}{p}}f(t)dt\right)^{p}dx\\
&\geq\int_{0}^{\infty}G^{-q}(x)\left[\left(\int_{0}^{x}t^{a}dt\right)^{\frac{p-1}{p}}\left(\int_{0}^{x}t^{-a(p-1)}f^{p}(t)dt\right)^{\frac{1}{p}}\right]^{p}dx\\
&=\int_{0}^{\infty}G^{-q}(x)\left(\int_{0}^{x}t^{a}dt\right)^{p-1}\int_{0}^{x}t^{-a(p-1)}f^{p}(t)dtdx\\
&=\frac{1}{(1+a)^{p-1}}\int_{0}^{\infty}t^{-a(p-1)}f^{p}(t)\int_{t}^{\infty}x^{(1+a)(p-1)}G^{-q}(x)dxdt\\
&\leq\frac{1}{(1+a)^{p-1}}\int_{0}^{\infty}t^{-a(p-1)}f^{p}(t)\left(\frac{t^{2}}{G(t)}\right)^{q}\int_{t}^{\infty}x^{(1+a)(p-1)-2q}dxdt\\
&=\frac{1}{((a+1)(1-p)+2q-1)(1+a)^{p-1}}\int_{0}^{\infty}\frac{(tf(t))^{p}}{G^{q}(t)}dt.
\end{split}
\end{equation}
\end{proof}
\begin{remark} Theorem \ref{t1} and Theorem \ref{t2} are an answer to an open problem proposed by B. Sroysang \cite{bs}.
\end{remark}

The other type is given by the following.
\begin{theorem}\label{Khaled} Let $f,g>0$ and non-decreasing on $(0,\infty)$. $F$ be as defined
by (\ref{1}) and $G$ defined by (\ref{001}). Let $\varphi,\psi$ are non-decreasing
,submultiplicative and convex. If $\varphi(f(x)\psi(g(x))$ is integrable, then the following inequality
\[
\int_{0}^{\infty}\varphi(F(x))\psi(G(x))x^{1-p}dx\leq\frac{1}{p-1}\int_{0}^{\infty}\varphi(f(x))\psi(g(x))x^{1-p}dx.\]
holds for all $p>1.$
\end{theorem}
\begin{proof} Since the functions $\varphi$ and $\psi$ are convex and submultiplicative, we obtain
\begin{center}
\begin{equation*}
\begin{split}
\int_{0}^{\infty}\frac{\varphi(F(x))\psi(G(x))x^{1-p}}{\varphi(x)\psi(x)}dx&=\int_{0}^{\infty}\frac{x^{1-p}}{\varphi(x)\psi(x)}\varphi\left(\int_{0}^{x}f(t)dt\right)\psi\left(\int_{0}^{x}f(t)dt\right)dx\\
&=\int_{0}^{\infty}\frac{x^{1-p}}{\varphi(x)\psi(x)}\varphi\left(\int_{0}^{x}f(t)dt\right)\psi\left(\int_{0}^{x}f(t)dt\right)dx\\
&\leq\int_{0}^{\infty}x^{1-p}\varphi\left(\frac{1}{x}\int_{0}^{x}f(t)dt\right)\psi\left(\frac{1}{x}\int_{0}^{x}g(t)dt\right)dx.
\end{split}
\end{equation*}
\end{center}
So, by Lemma \ref{009} and the previous inequality we get
\begin{center}
\begin{equation}\label{ww}
\int_{0}^{\infty}\frac{\varphi(F(x))\psi(G(x))x^{1-p}}{\varphi(x)\psi(x)}dx\leq\int_{0}^{\infty}x^{-1-p}\left(\int_{0}^{x}\varphi(f(t))dt\right)\left(\int_{0}^{x}\psi(g(t))dt\right)dx.
\end{equation}
\end{center}
Since the functions $\varphi,\;\psi,\;f\;\textrm{and} \;g$ are non-decreasing then the function $\varphi\circ f$ and $\psi\circ g$ are also non-decreasing and consequently we can applied  Lemma \ref{mi} where $p(x)=1,$ and by the inequality (\ref{ww}) we have 
\begin{equation}
\begin{split}
\int_{0}^{\infty}\frac{\varphi(F(x))\psi(G(x))x^{1-p}}{\varphi(x)\psi(x)}dx&\leq\int_{0}^{\infty}x^{-p}\left(\int_{0}^{x}\varphi(f(t))\psi(g(t))dt\right)dx\\
&=\int_{0}^{\infty}\varphi(f(t))\psi(g(t))\left(\int_{t}^{\infty}x^{-p}\varphi(x)\psi(x)dx\right)dt\\
&=\int_{0}^{\infty}\varphi(f(t))\psi(g(t))\left(\int_{t}^{\infty}x^{-p}dx\right)dt\\
&=\frac{1}{p-1}\int_{0}^{\infty}x^{1-p}\varphi(f(x)\psi(g(x))dx.
\end{split}
\end{equation}
\end{proof}
\begin{theorem}
Let $f\geq0,$ non-decreasing, $F$ is as defined by
(\ref{1}). Let $g>0,$ be a  continuous on $(0,\infty)$, $G$ is as defined
by (\ref{001}). Let $\phi\geq0,$ and non-decreasing, and $0<b\leq\infty.$ If $g$ is non-increasing,and $\phi\left(\frac{f(x)}{g(x)}\right)$ is integrable on $0<b\leq\infty,$  the following inequality is valid
\begin{equation}
\int_{0}^{b}\phi\left(\frac{F(x)}{G(x)}\right)dx\leq\int_{0}^{b}\phi\left(\frac{f(x)}{g(x)}\right)dx.
\end{equation}
In particular, by putting  $\phi(x)=x^{p},\, p\geq1,$
we obtain 
\begin{equation}\label{K01}
\int_{0}^{b}\left(\frac{F(x)}{G(x)}\right)^{p}dx\leq\int_{0}^{b}\left(\frac{f(x)}{g(x)}\right)^{p}dx.
\end{equation}
\end{theorem}
\begin{proof}
Since $g$ is non-increasing, $f$ and $\phi$ are non-decreasing, we get
\begin{equation*}
\begin{split}
\int_{0}^{b}\phi\left(\frac{F(x)}{G(x)}\right)dx&=\int_{0}^{b}\phi\left(\frac{1}{G(x)}\int_{0}^{x}f(t)dt\right)\\
&\leq\int_{0}^{b}\phi\left(\frac{xf(x)}{G(x)}\right)dx\\
&=\int_{0}^{b}\phi\left(\frac{xf(x)}{\int_{0}^{x}g(t)dt}\right)dx\\
&\leq\int_{0}^{b}\phi\left(\frac{xf(x)}{xg(x)}\right)dx.
\end{split}
\end{equation*}
\end{proof}
The next theorem is a generalization of Theorem 2.5 in \cite{ws}.
\begin{theorem}\label{tt2} Let $\varphi(x)\geq0$  be a twice differentiable function on $(0,\infty)$, convex, submultiplicative and $\varphi(0)=0.$ Let $q\in\mathbb{N}$ and $p>1.$ If $x^{2-p}\frac{\varphi(f(x))}{\varphi(x)}$ is integrable , then the following inequality 
\begin{equation}\label{MM}
\int_{0}^{\infty}x^{2-p}\frac{\varphi\left(x^{q}F(x)\right)}{\varphi^{q+2}(x)}dx\leq \frac{1}{p-1}\int_{0}^{\infty}x^{2-p}\frac{\varphi(f(x))}{\varphi(x)}dx.
\end{equation}
holds. 
\end{theorem}
\begin{proof} Since the function $\varphi(x)$ is convex, we have $\frac{\varphi(x)}{x}$ is non-decreasing by Lemma \ref{l3}. Then, by the Jensen's inequality,  Fubini theorem and since $\varphi$ is submultiplicative (In particular $\varphi(x^n)\leq \varphi^n(x)\;n\geq1 $),  we get
 \begin{equation*}\begin{split}
\int_{0}^{\infty}x^{2-p}\frac{\varphi\left(x^{q}F(x)\right)}{\varphi^{q+2}(x)}dx&=\int_{0}^{\infty}\frac{x^{2-p}}{\varphi^{q+2}(x)}\varphi\left(x^{q+1}\frac{F(x)}{x}\right)dx\\
&\leq\int_{0}^{\infty}\frac{x^{2-p}}{\varphi^{q+2}(x)}\varphi(x^{q+1})\varphi\left(\frac{F(x)}{x}\right)dx\\
&\leq\int_{0}^{\infty}\frac{x^{2-p}}{\varphi^{q+2}(x)}\varphi^{q+1}(x)\varphi\left(\frac{F(x)}{x}\right)dx\\
&=\int_{0}^{\infty}\frac{x^{2-p}}{\varphi(x)}\varphi\left(\frac{F(x)}{x}\right)dx\\
&=\int_{0}^{\infty}\frac{x^{2-p}}{\varphi(x)}\varphi\left(\frac{1}{x}\int_{0}^{x}f(t)dt\right)dx\\
&\leq\int_{0}^{\infty}\frac{x^{2-p}}{\varphi(x)}\left(\frac{1}{x}\int_{0}^{x}\varphi(f(t))dt\right)dx\\
&=\int_{0}^{\infty}\varphi(f(t))\left(\int_{t}^{\infty}x^{-p}\frac{x}{\varphi(x)}dx\right)dt\\
&\leq\int_{0}^{\infty}\varphi(f(t))\frac{t}{\varphi(t)}\left(\int_{t}^{\infty}x^{-p}dx\right)dt\\
&=\frac{1}{p-1}\int_{0}^{\infty}x^{2-p}\frac{\varphi(f(x))}{\varphi(x)}dx.
\end{split}
\end{equation*}
So, the proof of Theorem \ref{tt2} is complete. 
\end{proof}
\begin{theorem}
Let $f\geq0,$ non-decreasing  on $(0,\infty)$ and $F$ is as defined by (\ref{1}). Let $g>0,$ be a  non-decreasing  on $(0,\infty)$ and $G$ is as defined
by (\ref{001}). Let $\phi\geq0,$ and non-decreasing, and $0<a<b<\infty.$ If the function $\phi(f(x)g(x))$ is integrable on $[a,b]$, then 
\begin{equation}\label{K1000}
\int_{a}^{b}\phi\left(\frac{F(x)G(x)}{x^{2}}\right)dx\leq\int_{a}^{b}\phi\left(f(x)g(x)\right)dx.
\end{equation}
In particular, by putting $\phi(x)=x^{p},\;p\geq1$ and $g(x)=1$, we obtain
\begin{equation*}
\int_{a}^{b}\left(\frac{F(x)}{x}\right)^{p}dx\leq\int_{a}^{b}f^{p}(x)dx.
\end{equation*}
\end{theorem}
\begin{proof}
By the Chebyshev integral inequality  and  the assumption of the functions $f,g$ and $\phi$, and we consider the function $p(x)=1$ for all $x\in[a,b]$ we obtain
\begin{equation*}
\begin{split}
\int_{a}^{b}\phi\left(\frac{F(x)G(x)}{x^{2}}\right)dx&=\int_{a}^{b}\phi\left[\frac{1}{x^{2}}\left(\int_{0}^{x}f(t)dt\right)\left(\int_{0}^{x}g(t)dt\right)\right]\\
&\leq\int_{a}^{b}\phi\left[\frac{1}{x^{2}}\int_{0}^{x}1 dt\int_{0}^{x}f(t)g(t)dt\right]dx\\
&=\int_{a}^{b}\phi\left[\frac{1}{x}\int_{0}^{x}f(t)g(t)dt\right]dx\\
&\leq\int_{a}^{b}\phi\left[\frac{f(x)}{x}\int_{0}^{x}g(t)dt\right]dx\\
&\leq\int_{a}^{b}\phi\left[\frac{f(x)g(x)}{x}\int_{0}^{x}dt\right]dx\\
&=\int_{a}^{b}\phi\left(f(x)g(x)\right)dx.
\end{split}
\end{equation*}
\end{proof}
The following result concerns the converse inequality (\ref{K1000}).
\begin{theorem}
Let $f\geq0,$ non-decreasing, $F$ is as defined by (\ref{1}).Let $g>0,$  non-decreasing, $G$ is as defined
by (\ref{001}). Let $\phi\geq0,$ and non-increasing, and  $0<a<b<\infty,$  then
\begin{equation}\label{K1}
\int_{a}^{b}\phi\left(\frac{F(x)G(x)}{x^{2}}\right)dx\geq\int_{a}^{b}\phi\left(f(x)g(x)\right)dx.
\end{equation}
\end{theorem}
\begin{proof} By the Chebyshev integral inequality  and  the assumption of the functions $f,g$ and $\phi$, and we consider the function $p(x)=1$ for all $x\in[a,b]$ we obtain
\begin{equation*}
\begin{split}
\int_{a}^{b}\phi\left(\frac{F(x)G(x)}{x^{2}}\right)dx&=\int_{a}^{b}\phi\left[\frac{1}{x^{2}}\left(\int_{0}^{x}f(t)dt\right)\left(\int_{0}^{x}g(t)dt\right)\right]\\
&\geq\int_{a}^{b}\phi\left[\frac{1}{x^{2}}\int_{0}^{x}1\;dt\int_{0}^{x}f(t)g(t)dt\right]dx\\
&=\int_{a}^{b}\phi\left[\frac{1}{x}\int_{0}^{x}f(t)g(t)dt\right]dx\\
&\geq\int_{a}^{b}\phi\left[\frac{f(x)}{x}\int_{0}^{x}g(t)dt\right]dx\\
&\geq\int_{a}^{b}\phi\left[\frac{f(x)g(x)}{x}\int_{0}^{x}dt\right]dx\\
&=\int_{a}^{b}\phi\left(f(x)g(x)\right)dx.
\end{split}
\end{equation*}
\end{proof}

\vspace{2mm} \noindent \footnotesize
\vspace{2mm} \noindent \footnotesize
\begin{minipage}[b]{10cm}
Khaled Mehrez \\
Departement of Mathematics, University of Kairaouan,\\
ISSAT. Kasserine 1200, Tunisia.\\
Email: k.mehrez@yahoo.fr
\end{minipage}
\end{document}